\documentclass[11pt,oneside,english]{amsart}
\usepackage[latin9]{inputenc}
\usepackage{geometry}
\usepackage{amsthm}
\usepackage{amstext}
\usepackage{amssymb}
\usepackage{setspace}
\usepackage{esint}

\makeatletter
\numberwithin{equation}{section}
\numberwithin{figure}{section}
\theoremstyle{plain}
\newtheorem{theorem}{Theorem}[section]
\newtheorem{proposition}[theorem]{Proposition}
\theoremstyle{definition}
\newtheorem{example}[theorem]{Example}
\theoremstyle{remark}
\newtheorem{remark}[theorem]{Remark}
\theoremstyle{plain}
\newtheorem{corollary}[theorem]{Corollary}

\usepackage{babel}

\makeatother

\usepackage{babel}

\begin{document}

\title{Growth and Integrability of Fourier Transforms on Euclidean Space}

\author{William O. Bray}

\address{Department of Mathematics, Missouri State University, Springfield,
MO 65837}

\email{wbray@missouristate.edu}


\begin{abstract}
Abstract. A fundamental theme in classical Fourier analysis relates smoothness properties of functions 
to the growth and/or integrability of their Fourier transform. By using a suitable class of $L^{p}-$
 multipliers, a rather general inequality controlling the size of Fourier transforms for large and small 
 argument is proved. As consequences, quantitative Riemann-Lebesgue estimates are obtained 
 and an integrability result for the Fourier transform is developed extending ideas used by Titchmarsh 
 in the one dimensional setting.
 \end{abstract}
\maketitle

\section{Introduction}

A classical theme in Fourier analysis relates smoothness of functions
to the growth and/or integrability of their Fourier components. In
this vein, the following inequality was proved in \cite{BP1}. Here
the one dimensional Fourier transform is defined for integrable functions
as
\[
\widehat{f}(\lambda)=\int_{\mathbb{R}}f(x)\, e^{-i\lambda x}dx,
\]
and extended to $L^{p}(\mathbb{R})$ in the usual fashion; 
\[
\Omega_{p}[f](t)=\sup_{0<h<t}\Vert f(\cdot+h)+f(\cdot-h)-2f(\cdot)\Vert_{p}
\]
is an $L^{p}-$modulus of continuity based on symmetric differences;
and $p'$ denotes the H\={o}lder conjugate index ($\frac{1}{p}+\frac{1}{p'}=1$).
\begin{theorem}
\label{thm:BP-1d} Let $1\leq p\leq2$. Then there is a constant $c_{p}>0$
such that for all $f\in L^{p}(\mathbb{R})$,
\begin{itemize}
\item when $p=1,$ 
\[
\sup_{\lambda}\left[\min\{1,(\lambda t)^{2}\}\vert\widehat{f}(\lambda)\vert\right]\leq c_{1}\Omega_{1}[f](t);
\]

\item when $1<p\leq2$, 
\[
\left[\int_{\mathbb{R}}\min\{1,(\lambda t)^{2p'}\}\vert\widehat{f}(\lambda)\vert^{p'}d\lambda\right]^{1/p'}\leq c_{p}\Omega_{p}[f](t).
\]

\end{itemize}
\end{theorem}
The significance of this inequality stems from the presence of the
minimum function that gives control over the Fourier transform for
small and large $\lambda$. Indeed, for $1<p\leq2$, the inequality
may be rewritten
\[
\underset{\text{{large}\,\ensuremath{\lambda}}}{\underbrace{\int_{|\lambda|\geq1/t}\vert\widehat{f}(\lambda)\vert^{p'}d\lambda}\;+\;}\underset{\text{small}\,\lambda}{\underbrace{t^{2p'}\int_{|\lambda|<1/t}\lambda^{2p'}\vert\widehat{f}(\lambda)\vert^{p'}d\lambda}}\leq c_{p}^{p'}\Omega_{p}^{p'}[f](t).
\]
As shown in \cite{BP1}, the estimate for large $\lambda$ yields
a qualitative Riemann-Lebesgue lemma. On the other hand, from the
estimate for small $\lambda$, the following integrability theorem
can be proved.
\begin{proposition}
\label{prop:B-1d-int}Let $1\leq p\leq2$ and $f\in L^{p}(\mathbb{R})$.
If for some $0<\alpha\leq2$ 
\begin{equation}
\Vert f(\cdot+t)+f(\cdot-t)-2f(\cdot)\Vert_{p}=O(t^{\alpha}),\label{eq:lip-1d}
\end{equation}
then $\widehat{f}\in L^{\beta}(\mathbb{R})$ provided 
\[
\frac{p}{p+\alpha p-1}<\beta\leq p'.
\]
In particular, if (\ref{eq:lip-1d}) holds for $\alpha>1/p$, then
$\widehat{f}\in L^{1}(\mathbb{R})$ and Fourier inversion holds a.e.%
\footnote{Fourier inversion actually holds everywhere. The Fourier inversion
integral defines a continuous function which is equal to $f$ a.e.
This remark also applies to Corollaries \ref{cor:ae FT}, \ref{cor:szasz},
and \ref{cor:ae FS}.%
}
\end{proposition}
The proof of this follows from techniques used later for its generalization
to higher dimensions. This result is an extension of one given in
Titchmarsh \cite[Theorem 84]{T} where first order differences were
used instead of the second order difference in (\ref{eq:lip-1d})
and consequently the restrictions are $0<\alpha\leq1$ and $1<p\leq2$.

Theorem \ref{thm:BP-1d} has an extension to higher dimensions \cite{BP1,BP2}
stated below; the modulus of continuity now is based on the the spherical
mean operator defined by
\[
M^{t}f(x)=\frac{1}{\omega_{n-1}}\int_{S^{n-1}}f(x+t\omega)d\omega,
\]
where $S^{n-1}$ is the unit sphere in $\mathbb{R}^{n}$, $\omega_{n-1}$
is its measure, and $d\omega$ is induced Lebesgue measure.
\begin{theorem}
\label{BP-thm}Let $n\geq2$, let $1\leq p\leq2$. Then there is a
constant $c_{p}>0$ such that for all $f\in L^{p}(\mathbb{R}^{n})$,
\begin{itemize}
\item when $p=1$, 
\[
\sup_{\xi}\left[\min\{1,(t|\xi|)^{2}\}\vert\widehat{f}(\xi)\vert\right]\leq c_{1}\Vert M^{t}f(\cdot)-f(\cdot)\Vert_{1};
\]

\item when $1<p<2$,
\[
\left[\int_{\mathbb{R}^{n}}\min\{1,(t|\xi|)^{2p'}\}\vert\widehat{f}(\xi)\vert^{p'}d\xi\right]^{1/p'}\leq c_{p}\Vert M^{t}f(\cdot)-f(\cdot)\Vert_{p};
\]

\item when $p=2$,
\[
\left[\int_{\mathbb{R}^{n}}\min\{1,(t|\xi|)^{4}\}\vert\widehat{f}(\xi)\vert^{2}d\xi\right]^{1/2}\asymp\Vert M^{t}f(\cdot)-f(\cdot)\Vert_{2}.
\]
(Here $\asymp$ means the left hand side is bounded above and below
by a positive constant times the right hand side.)
\end{itemize}
\end{theorem}
Aside from the use of the Hausdorff-Young inequality/Plancherel theorem,
the proof of Theorem \ref{BP-thm} depends on two other key ideas:
\begin{itemize}
\item The multiplier identity
\[
\widehat{M^{t}f}(\xi)=\widehat{f}(\xi)\, j_{\nu}(t|\xi|),
\]
where $j_{\nu}$ is the spherical Bessel function of order $\nu=\frac{n-2}{2}$,
\[
j_{\nu}(r)=2^{\nu}\Gamma(\nu+1)r^{-\nu}J_{\nu}(r),
\]
$J_{\nu}$ being Bessel function of the first kind.
\item The estimate (for $\alpha>-1/2$)
\[
1-j_{\alpha}(\lambda)\asymp\min\{1,\lambda^{2}\},
\]
derived from the Mehler representation of Bessel functions (see \cite{BP2}).
\end{itemize}
\noindent The use of spherical means in Theorem \ref{BP-thm} was
motivated by the work of Gioev \cite{G}. Defining a modulus of continuity
using the spherical mean operator seems to have first been used in
the work of Platonov \cite{P} where a generalization of a different
result of Titchmarsh was proved. Theorem \ref{BP-thm} (and a generalization)
was also obtained by Ditzian \cite{D1} using related methods and
from an approximation theory perspective. From that work, it is clear
that the term $\Vert M^{t}f(\cdot)-f(\cdot)\Vert_{p}$ measures second
order smoothness and can be thought of as an analog for $L^{p}(\mathbb{R}^{n})$,
$n\geq2$, of the second order difference operator used in Theorem
\ref{thm:BP-1d}. In later work, Ditzian \cite{D2} obtained a variation
of Theorem \ref{BP-thm} replacing the Hausdorff-Young inequality
by the Hardy-Littlewood inequality. Another variation was obtained
by Gorbachev and Tikhonov \cite{GT} making use of Pick's inequality.
Further comments regarding these works and the relation to the results
of the present paper will be given later.

The focus of the current paper is on generalizations of Theorem \ref{BP-thm}
and on extensions of Proposition \ref{prop:B-1d-int} to higher dimensions.
Section 2 gives a wide class of multiplier operators for which an
analog of Theorem \ref{BP-thm} is valid as well as variations in
the vein of \cite{D2,GT}. In section 3, the analog of Proposition
\ref{prop:B-1d-int} is developed. Finally, a couple of variations
are presented in section 4. Many of the results developed here have
analogs in rank one symmetric spaces and more generally, Damek-Ricci
spaces; details will appear in a sequel to this paper.

\section{Generalizations \& Variations of Theorem \ref{BP-thm}}

Let $\mu$ be a finite Borel measure on $\mathbb{R}^{n}$, $n\geq2$,
we assume $\int_{\mathbb{R}^{n}}d\mu=1$. The total variation of $\mu$
is denoted $\Vert\mu\Vert_{\mathcal{M}}=\int_{\mathbb{R}^{n}}d|\mu|$.
Using convolution, $\mu$ defines a Fourier multiplier operator on
$L^{1}(\mathbb{R}^{n})$; on the Fourier transform side, the multiplier
is given by the Fourier transform of the measure $\widehat{\mu}(\xi)$.
For $t>0$, the dilation of $\mu$ is defined via the following natural
formula:
\[
f\in L^{1}(\mathbb{R}^{n}),\;\int_{\mathbb{R}^{n}}f(x)d\mu_{t}(x)=\int_{\mathbb{R}^{n}}f(t\, x)d\mu(x).
\]
The measure $\mu_{t}$ defines a Fourier multiplier operator on $L^{1}(\mathbb{R}^{n})$
through convolution:
\[
M_{\mu}^{t}f(x)=(f*\mu_{t})(x)=\int_{\mathbb{R}^{n}}f(x-ty)d\mu(y).
\]
It is immediate that $\Vert M_{\mu}^{t}f\Vert_{1}\leq\Vert f\Vert_{1}\Vert\mu\Vert_{\mathcal{M}}$.
From standard theory (e.g. \cite{Graf}) the operator is also a Fourier
multiplier on $L^{p}(\mathbb{R}^{n})$ for all $1\leq p\leq\infty$;
the corresponding multiplier is given by $\widehat{\mu_{t}}(\xi)=\widehat{\mu}(t\xi)$.
The following result is an extension of classical approximate identity
ideas.
\begin{proposition}
\label{prop: approx-ident}Let $\mu$ be as above and $1\leq p<\infty$.
Then for all $f\in L^{p}(\mathbb{R}^{n})$, $\Vert M_{\mu}^{t}f-f\Vert_{p}\rightarrow0$
as $t\rightarrow0$.
\end{proposition}
\begin{proof}
Consider the case $p>1$; $p=1$ is similar. Applying H\={o}lder's inequality,
\[
\vert M_{\mu}^{t}f(x)-f(x)\vert\leq\Vert\mu\Vert_{\mathcal{M}}^{1/p'}\left(\int_{\mathbb{R}^{n}}\vert f(x-ty)-f(x)\vert^{p}d\vert\mu\vert(y)\right)^{1/p}.
\]
Raising both sides to the power $p$, integrating over $\mathbb{R}^{n}$,
and interchanging orders of integration yields:
\begin{equation}
\Vert M_{\mu}^{t}f-f\Vert_{p}^{p}\leq\Vert\mu\vert_{\mathcal{M}}^{p-1}\int_{\mathbb{R}^{n}}\left[\int_{\mathbb{R}^{n}}\vert f(x-ty)-f(x)\vert^{p}dx\right]d\vert\mu\vert(y).\label{eq:app-iden}
\end{equation}
Let $\omega_{p}[f](h)=\Vert f(\cdot+h)-f(\cdot)\Vert_{p}$ be the
$L^{p}-$modulus of continuity; it is known that $\omega_{p}[f](h)\leq2\Vert f\Vert_{p}$
and $\omega_{p}[f](h)\rightarrow0$ as $|h|\rightarrow0$. Let $R>0$
and split the outer integral in (\ref{eq:app-iden}) into two integrals,
$I_{1}(t,R)$ over $|y|<R$ and $I_{2}(t,R)$ over $|y|\geq R$, respectively.
Let $\varepsilon>0$. We have
\[
I_{2}(t,R)\leq(2\Vert f\Vert_{p})^{p}\Vert\mu\Vert_{\mathcal{M}}^{p-1}\int_{|y|\geq R}d\vert\mu\vert(y);
\]
hence there is an $R>0$ such that $I_{2}(t,R)<(\varepsilon/2)^{p}$.
Choose $\delta>0$ such that if $|h|<\delta$, then $\omega_{p}[f](h)<\varepsilon/(2\Vert\mu\Vert_{\mathcal{M}}^{p})$.
Since $R$ has been fixed, choose $t_{0}>0$ such that $t_{0}R<\delta$.
Then for $t<t_{0}$, 
\[
I_{2}(t,R)=\Vert\mu\Vert_{\mathcal{M}}^{p-1}\int_{|y|<R}\omega_{p}[f](ty)^{p}d\vert\mu\vert(y)<\left(\frac{\varepsilon}{2}\right)^{p}.
\]
Putting the estimates together completes the proof.
\end{proof}
Setting the stage for putting Theorem \ref{BP-thm} from the introduction
in general form, let $\mu$ be a finite Borel measure on $\mathbb{R}^{n}$
with $\int_{\mathbb{R}^{n}}d\mu=1$ and whose Fourier multiplier satisfies
the estimate 
\begin{equation}
|1-\widehat{\mu}(\xi)|\asymp\min\{1,|\xi|^{2\sigma}\},\label{eq:meas-estimate}
\end{equation}
for some $\sigma>0$. The class of such measures is denoted $\mathcal{K}_{\sigma}=\mathcal{K}_{\sigma}(\mathbb{R}^{n})$.
\begin{theorem}
\label{thm:gen-estimates}Let $1\leq p\leq2$ and let $\mu\in\mathcal{K}_{\sigma}$.
Then there is a constant $c_{p}>0$ such that for all $f\in L^{p}(\mathbb{R}^{n})$,
\begin{enumerate}
\item when $p=1$,
\[
\sup_{\xi\in\mathbb{R}^{n}}\left[\min\{1,(t|\xi|)^{2\sigma}\}\vert\widehat{f}(\xi)\vert\right]\leq c_{1}\Vert M_{\mu}^{t}f-f\Vert_{1};
\]

\item when $1<p<2$,
\[
\left[\int_{\mathbb{R}^{n}}\min\{1,(t|\xi|)^{2\sigma p'}\}\vert\widehat{f}(\xi)\vert^{p'}d\xi\right]^{1/p'}\leq c_{p}\Vert M_{\mu}^{t}f-f\Vert_{p};
\]

\item when $p=2$,
\[
\left[\int_{\mathbb{R}^{n}}\min\{1,(t|\xi|)^{4\sigma}\}\vert\widehat{f}(\xi)\vert^{2}d\xi\right]^{1/2}\asymp\Vert M_{\mu}^{t}f-f\Vert_{2}.
\]

\end{enumerate}
\end{theorem}
\begin{proof}
We give the proof of item 2. Since $\widehat{M_{\mu}^{t}f}(\xi)=\widehat{f}(\xi)\widehat{\mu}(t\xi)$
we have
\[
\widehat{(M_{\mu}^{t}f-f)}(\xi)=[\widehat{\mu}(t\xi)-1]\widehat{f}(\xi).
\]
Applying the Hausdorff-Young inequality we obtain
\[
\left[\int_{\mathbb{R}^{n}}\vert1-\widehat{\mu}(t\xi)\vert^{p'}\vert\widehat{f}(\xi)\vert^{p'}d\xi\right]^{1/p'}\leq c_{p}\Vert M_{\mu}^{t}f-f\Vert_{p}.
\]
The inequality follows by from the estimate (\ref{eq:meas-estimate}).
Item 1 follows in a similar fashion using the $L^{1}-$estimate $\Vert\widehat{f}\Vert_{\infty}\leq\Vert f\Vert_{1}$,
and item 3 follows using the Plancherel theorem.
\end{proof}
Below are several examples that provide concrete realization of this
result.
\begin{example}
\label{sph-mean-examp}Theorem \ref{BP-thm} from the introduction
is a corollary of the above theorem as follows. Let $d\mu=d\omega/\omega_{n-1}$,
where $d\omega$ is the usual surface measure on the unit sphere in
$\mathbb{R}^{n}$. Then $M_{\mu}^{t}=M^{t}$, the spherical mean operator.
In this case $\widehat{\mu}(\xi)=j_{\nu}(|\xi|)$ where $\nu=\frac{n-2}{2}$,
and we have the estimate $1-j_{\nu}(r)\asymp\min\{1,r^{2}\}$, so
$\mu\in\mathcal{K}_{1}$. For later reference, the following identity
was used in \cite{BP1} to derive this estimate: for $\alpha>-1/2$,
\begin{equation}
1-j_{\alpha}(\lambda)=\frac{4\Gamma(\alpha+1)}{\sqrt{\pi}\Gamma(\alpha+1/2)}\int_{0}^{1}(1-s^{2})^{\alpha-1/2}\sin^{2}\frac{\lambda s}{2}ds.\label{eq:mehler}
\end{equation}
This example can be extended by iteration as follows. If $\mu^{(l)}$
denotes the convolution of $\mu$ with itself $l-$times, the corresponding
multiplier is $(j_{\nu}(|\xi|))^{l}$ and satisfies the same estimate
as above for $\widehat{\mu}(\xi)$. The corresponding operator is
the $l^{\text{th}}-$iterate of the spherical mean operator, denoted
$M_{l}^{t}$.
\end{example}

\begin{example}
Let the Lebesgue measure of a set $A$ be denoted $|A|$, and let
$\Omega_{n}$ be the measure of the unit ball $B(0,1)$. For a measurable
set $A\subset\mathbb{R}^{n}$, define $\mu(A)=|A\cap B(0,1)|/\Omega_{n}$.
To this measure, corresponds the averaging operator
\[
\mathcal{A}^{t}f(x)=\int_{\mathbb{R}^{n}}f(x-ty)d\mu(y)=\frac{1}{\Omega_{n}}\int_{B(0,1)}f(x-ty)dy.
\]
An easy computation shows that $\widehat{\mu}(\xi)=j_{\nu+1}(|\xi|)$
where $\nu=\frac{n-2}{2}$ and hence, $\mu\in\mathcal{K}_{1}$ from
previous considerations.
\end{example}

\begin{example}
\label{cube-examp}Let $C$ be the surface of the cube in $\mathbb{R}^{n}$
centered at the origin whose $2n$ faces each have area one. Here
we take $d\mu=dS/(2n)$, where $dS$ is the surface measure on the
boundary $\partial C$ of $C$ induced by Lebesgue measure. A simple,
albeit tedious, calculation gives the multiplier:
\[
\widehat{\mu}(\xi)=\frac{1}{2n}\int_{\partial C}e^{-i(\xi\cdot x)}dS(x)=\frac{1}{2^{n-1}n}\sum_{k=1}^{n}\cos\frac{\xi_{k}}{2}\prod_{l=1,l\neq k}^{n}\frac{\sin\frac{\xi_{l}}{2}}{\xi_{l}}.
\]
In this case we have $\vert1-\widehat{\mu}(\xi)\vert\asymp\min\{1,\vert\xi\vert^{2}\}$
and $\mu\in\mathcal{K}_{1}$.
\end{example}

The above examples are unified and generalized as follows. Let $K\subset\mathbb{R}^{n}$
be a compact connected symmetric ($-x\in K$ if $x\in K$) set with
non-empty interior and whose boundary $S$ is a piecewise smooth regular
surface. The latter means that $S=\cup_{l=1}^{m}S_{l}$, where each
$S_{l}$ is a smooth surface that is given as the level surface of
a smooth function $F_{l}(x)=0$ whose gradient never vanishes on $S_{l}$. 
\begin{proposition}
\label{prop:surf-meas}Let $K\subset\mathbb{R}^{n}$ be a compact
set as above.
\begin{enumerate}
\item Let $\mu$ be the normalized surface measure on $S$ induced by Lebesgue
measure. Then $\mu\in\mathcal{K}_{1}$.
\item Let $\mu$ be the measure on $\mathbb{R}^{n}$ defined by $\mu(A)=|A\cap K|/|K|$,
for any measurable set $A\subset\mathbb{R}^{n}$. Then $\mu\in\mathcal{K}_{1}$.
\end{enumerate}
\end{proposition}
\begin{proof}
We prove (1); the proof of (2) is similar. Denote normalized surface
measure on $S$ by $dS$. Since $K$ is a symmetric set,
\[
\widehat{\mu}(\xi)=\int_{S}\cos\xi\cdot x\, dS(x),
\]
and consequently
\[
1-\widehat{\mu}(\xi)=2\int_{S}\sin^{2}\frac{\xi\cdot x}{2}\, dS(x).
\]
Clearly $\vert1-\widehat{\mu}(\xi)\vert\leq2$. We claim that $1-\widehat{\mu}(\xi)=0$
only at the origin and consequently, given $r>0$, there is a constant
$c>0$ (dependent on $r$) such that $\vert1-\widehat{\mu}(\xi)\vert\geq c$
provided $\vert\xi\vert>r$. To prove the claim in the case where
$S$ is a smooth surface, suppose $\xi\neq0$ such that $1-\widehat{\mu}(\xi)=0$.
Since $x\rightarrow\frac{\xi\cdot x}{2}$ is continuous on $S$, $\sin\frac{\xi\cdot x}{2}=0$
for all $x\in S$. However, this implies that $S$ if contained in
some plane $\xi\cdot x=2\pi k$, where $k$ is an integer, a contradiction.
In the case where $S=\cup_{m}S_{m}$, where each $S_{m}$ is smooth,
the above ideas imply that the entire surface is contained in a union
of parallel planes, again a contradiction. To obtain the rest of the
estimate, take $R>0$ such that $K\subset\{x\,:\,|x|<R\}$ and suppose
$|\xi|<\pi/R$. Then from the inequality $\frac{2}{\pi}|u|\leq|\sin u|\leq u$
for $|u|\leq\pi/2$, it follows that for all $x\in S$,
\[
\frac{1}{\pi^{2}}(\xi\cdot x)^{2}\leq\sin^{2}\frac{\xi\cdot x}{2}\leq\frac{1}{4}(\xi\cdot x)^{2}.
\]
Integrating over $S$, there are positive constants $c$ and $c'$
such that
\[
c\, N^{2}(\xi)\leq\vert1-\widehat{\mu}(\xi)\vert\leq c'N^{2}(\xi),
\]
where $N(\xi)=\left[\int_{S}(\xi\cdot x)^{2}dS(x)\right]^{1/2}$.
It is easy to see that $\xi\rightarrow N(\xi)$ is a norm on $\mathbb{R}^{n}$
and since all norms on $\mathbb{R}^{n}$ are equivalent, there are
positive constants $d$ and $d'$ such that $d|\xi|\leq N(\xi)\leq d'|\xi|$.
Thus, for $|\xi|<\pi/R$, there are constants $c_{1}$ and $c_{2}$
such that
\[
c_{1}|\xi|^{2}\leq\vert1-\widehat{\mu}(\xi)\vert\leq c_{2}|\xi|^{2}
\]
and the result follows.
\end{proof}

The proposition and its proof have easy generalization by multiplying
the measures in (1) or (2) by suitable functions. For example, the
result given in (1) is valid for measures of the form $d\mu=\varphi dS,$
where $\phi$ is a non-negative continuous function on $\mathbb{R}^{n}$
with $\int_{S}\varphi dS=1$.

Other examples of measures in some $\mathcal{K}_{\sigma}$ come from
classical approximate identities. The following example is illustrative.
\begin{example}
\label{approx-ident-examp}Let $d\mu=H(x)dx$, where
\[
H(x)=\frac{1}{2^{n}\pi^{n/2}}e^{-|x|^{2}/4}
\]
(notice that $\int_{\mathbb{R}^{n}}H(x)dx=1$). In this case the multiplier
is $\widehat{H}(\xi)=e^{-|\xi|^{2}}$ and the estimate $1-e^{-|\xi|^{2}}\asymp\min\{1,|\xi|^{2}\}$
is elementary. The operator of interest is given by: $M_{\mu}^{t}f(x)=(f*H_{t})(x)$.
Comparing the conclusion of Theorem \ref{thm:gen-estimates} in the
$L^{2}$ case for the Examples \ref{sph-mean-examp} and \ref{approx-ident-examp}
above leads to the following interesting conclusion: for $f\in L^{2}(\mathbb{R}^{n})$,
\[
\Vert f*H_{t}-f\Vert_{2}\asymp\Vert M^{t}f-f\Vert_{2}.
\]
In other words, the approximations $f*H_{t}$ and $M^{t}f$ have the
same rate of approximation in $L^{2}-$norm.\\
Notice that if we dilate using $t^{1/2}$ instead of $t$, then $H_{t^{1/2}}(x)$
is the heat kernel, $H_{t^{1/2}}*f$ is the solution of the heat equation
$u_{t}=\Delta u$ ($\Delta$ being the Laplacian) with initial data
$u(x,0)=f(x)$, and further, the estimates in the theorem must be
modified accordingly.
\end{example}

For the results of section 3, measures are needed which lead to higher
values of $\sigma$ than in the previous examples. The following two
results effectively achieve this. The first is based on the iterate
$(I-M_{\mu}^{t})^{l}$ for appropriate measures $\mu$.
\begin{proposition}
\label{prop:binomial powers}Let $\mu$ be as in Proposition \ref{prop:surf-meas}
and $l\geq1$. Define
\[
\mu'=\sum_{k=1}^{l}(-1)^{k+1}\binom{l}{l-k}\mu^{(k)},
\]
where $\mu^{(k)}$ is convolution of $\mu$ with itself $k-$times.
Then $\mu'\in\mathcal{K}_{l}$ and moreover, the corresponding operator
is given by
\[
M_{\mu'}^{t}f(x)=\sum_{k=1}^{l}(-1)^{k}\binom{l}{l-k}M_{\mu^{(k)}}^{t}f(x).
\]
\end{proposition}
\begin{proof}
From the binomial theorem,
\begin{align*}
1-\widehat{\mu'}(\xi) & =1-\sum_{k=1}^{l}(-1)^{k+1}\binom{l}{l-k}\widehat{\mu}^{k}(\xi)\\
 & =\left(1-\widehat{\mu}(\xi)\right)^{l}.
\end{align*}
The result now follows as $\mu\in\mathcal{K}_{1}$.
\end{proof}

An alternative approach to constructing measures in $\mathcal{K}_{\sigma}$
with $\sigma>1$ generalizes operators used in \cite{DD} for problems
in approximation theory. These operators are described as follows.

\begin{example}
\label{Dai-Ditzian}The identity (\ref{eq:mehler}) was originally
used in obtaining the estimate $1-j_{\alpha}(\lambda)\asymp\min\{1,\lambda^{2}\}$.
In particular, the presence of the squared sine term on the right
hand side is key. The idea here is to develop an operator where this
term is replaced by sine to a higher power and leads to and motivates
operators introduced by Dai and Ditzian \cite{DD}. Let $l\geq1$
be an integer. Then using Euler's identity, the binomial theorem,
and standard manipulation yields the trigonometric identity \cite{DD}
\begin{align}
4^{l}\binom{2l}{l}^{-1}\sin^{2l}\frac{y}{2} & =1-2\binom{2l}{l}^{-1}\sum_{k=1}^{l}(-1)^{k+1}\binom{2l}{l-k}\cos ky\label{eq:trig identity}\\
 & =1-v_{l}(y).\nonumber 
\end{align}
For $\alpha>-1/2$, define
\begin{align*}
j_{\alpha,l}(\lambda) & =\frac{2\Gamma(\alpha+1)}{\sqrt{\pi}\Gamma(\alpha+1/2)}\int_{0}^{1}(1-y^{2})^{\alpha-1/2}v_{l}(\lambda y)dy\\
 & =2\binom{2l}{l}^{-1}\sum_{k=1}^{l}(-1)^{k+1}\binom{2l}{l-k}j_{\alpha}(k\lambda).
\end{align*}
With $\alpha=\nu=\frac{n-2}{2}$, this is the Fourier multiplier corresponding
to the measure
\[
\mu=2\binom{2l}{l}^{-1}\sum_{k=1}^{l}(-1)^{k+1}\binom{2l}{l-k}\omega_{k},
\]
where $\omega$ is normalized surface measure on the unit sphere and
$\omega_{k}$ is its dilation by $k$. The corresponding operator
is given by
\[
V_{l}^{t}f(x)=2\binom{2l}{l}^{-1}\sum_{k=1}^{l}(-1)^{k+1}\binom{2l}{l-k}M^{kt}f(x).
\]
This operator is precisely the one introduced in \cite{DD} for problems
in approximation theory and used in \cite{D1,D2,GT} in obtaining
generalization and variations of Theorem \ref{BP-thm}. In this case
the multiplier estimate is $\vert1-\widehat{\mu}(\xi)\vert\asymp\min\{1,|\xi|^{2l}\}$
as follows from the integral defining $j_{\nu,l}$ and the above trigonometric
identity.
\end{example}

\begin{proposition}
\label{prop:DD}Let $\mu$ be as in Proposition \ref{prop:surf-meas}
and $l\geq1$. Define
\[
\mu'=2\binom{2l}{l}^{-1}\sum_{k=1}^{l}(-1)^{k+1}\binom{2l}{l-k}\mu_{k},
\]
where $\mu_{k}$ is dilation of $\mu$ by $k$. Then $\mu'\in\mathcal{K}_{l}$
and the corresponding operator is given by
\[
M_{\mu'}^{t}f(x)=2\binom{2l}{l}^{-1}\sum_{k=1}^{l}(-1)^{k+1}\binom{2l}{l-k}M_{\mu}^{kt}f(x).
\]
\end{proposition}
\begin{proof}
We have
\begin{align*}
\widehat{\mu'}(\xi) & =2\binom{2l}{l}^{-1}\sum_{k=1}^{l}(-1)^{k+1}\binom{2l}{l-k}\widehat{\mu}(k\xi)\\
 & =\int_{\mathbb{R}^{n}}\left[2\binom{2l}{l}^{-1}\sum_{k=1}^{l}(-1)^{k+1}\binom{2l}{l-k}\cos k(x\cdot\xi)\right]d\mu(x)
\end{align*}
and hence
\[
1-\widehat{\mu'}(\xi)=4^{l}\binom{2l}{l}^{-1}\int_{\mathbb{R}^{n}}\sin^{2l}\frac{x\cdot\xi}{2}d\mu(x).
\]
The result follows by applying ideas used in the proof of Proposition
\ref{prop:surf-meas}.
\end{proof}

\begin{remark}
The right hand side of the estimates given in Theorem \ref{thm:gen-estimates}
take the form $\Vert(I-M_{\mu}^{t})^{l}f\Vert_{p}$ and $\Vert M_{\mu'}^{t}f-f\Vert_{p}$
in the case of the measures given by Proposition \ref{prop:binomial powers}
and Proposition \ref{prop:DD}, respectively. The advantage of the
latter lies in the fact that the operator $M_{\mu'}^{t}$ is a linear
combination of dilates of a single operator.
\end{remark}
This section concludes with the following result generalizing Theorem
\ref{thm:gen-estimates} and results given in \cite{D2,GT} for the
operators $V_{l}^{t}$ of Example \ref{Dai-Ditzian}.
\begin{theorem}
\label{thm:picks}Let $\mu\in\mathcal{K}_{\sigma}$. 
\begin{enumerate}
\item Let $1<p\leq2$, $p\leq q\leq p'$, and $f\in L^{p}(\mathbb{R}^{n})$.
Then for all $f\in L^{p}(\mathbb{R}^{n})$, $\vert\xi\vert^{n(1-\frac{1}{p}-\frac{1}{q})}\widehat{f}(\xi)\in L^{q}(\mathbb{R}^{n})$
and there is a constant $c_{p,q}>0$ such that
\[
\left[\int_{\mathbb{R}^{n}}\min\{1,(t|\xi|)^{2\sigma q}\}\vert\xi\vert^{qn(1-\frac{1}{p}-\frac{1}{q})}\vert\widehat{f}(\xi)\vert^{q}d\xi\right]^{1/q}\leq c_{p,q}\Vert M_{\mu}^{t}f-f\Vert_{p}.
\]

\item Let $2\leq p<\infty$ and let $q>1$ with $\max\{q,q'\}\leq p$. If
$f\in L^{p}(\mathbb{R}^{n})$ with $|\xi|^{n(1-\frac{1}{p}-\frac{1}{q})}\widehat{f}(\xi)\in L^{q}(\mathbb{R}^{n})$,
then there is a constant $c_{p,q}>0$ such that
\[
\Vert M_{\mu}^{t}f-f\Vert_{p}\leq c_{p,q}\left[\int_{\mathbb{R}^{n}}\min\{1,(t|\xi|)^{2\sigma q}\}\vert\xi\vert^{qn(1-\frac{1}{p}-\frac{1}{q})}\vert\widehat{f}(\xi)\vert^{q}d\xi\right]^{1/q}.
\]

\end{enumerate}
\end{theorem}
The proof uses the following special cases of Pick's inequality \cite{BH,DD}
instead of the Hausdorff-Young theorem. Under the assumptions of item
(1) above,
\[
\left[\int_{\mathbb{R}^{n}}\vert\xi\vert^{qn(1-\frac{1}{p}-\frac{1}{q})}\vert\widehat{f}(\xi)\vert^{q}d\xi\right]^{1/q}\leq c_{p,q}\Vert f\Vert_{p}.
\]
In the case $q=p'$, this is just the Hausdorff-Young inequality;
the case $q=p$, it is the Hardy-Littlewood inequality. Likewise,
under the assumptions of item (2), Pick's inequality takes the form
\[
\Vert f\Vert_{p}\leq c_{p,q}\left[\int_{\mathbb{R}^{n}}\vert\xi\vert^{qn(1-\frac{1}{p}-\frac{1}{q})}\vert\widehat{f}(\xi)\vert^{q}d\xi\right]^{1/q}.
\]
In the case where $q=p'$, this is the dual form of the Hausdorff-Young
inequality. To obtain the conclusions of the theorem, we apply the
above inequalities to $M_{\mu}^{t}f-f$ and proceed as in Theorem
\ref{thm:gen-estimates}.
\begin{remark}
\label{connection-with-approximation}For the operator $V_{l}^{t}$
of Example \ref{Dai-Ditzian}, the above result was proved in \cite{D2}
in the case $q=p$; the cases $p\leq q\leq p'$ were considered in
\cite{GT} for the same operators. In both cases the proof was based
on the connection between the operator and $K-$functionals associated
with the Laplacian. Specifically, in \cite{DD} it was shown that
\begin{equation}
\Vert V_{l}^{t}f-f\Vert_{p}\asymp K_{l}(f,\Delta^{l},t^{2l})_{p},\label{eq:K-functional}
\end{equation}
where $\Delta$ is the Laplacian on $\mathbb{R}^{n}$ and
\[
K_{l}(f,\Delta^{l},t^{2l})_{p}=\inf\left\{ \Vert f-g\Vert_{p}+t^{2l}\Vert\Delta^{l}g\Vert_{p}\right\} ,
\]
the infinum taken over all $g\in L^{p}(\mathbb{R}^{n})$ such that
$\Delta^{l}g\in L^{p}(\mathbb{R}^{n})$. As such, the $K-$functional
gives a gauge on the order of smoothness of approximations and the
equivalence (\ref{eq:K-functional}) gives the same interpretation
to the differences $\Vert V_{l}^{t}f-f\Vert_{p}$. The proofs given
in \cite{D1,D2,GT} of our results in the case of the operators $V_{l}^{t}$
depend on (\ref{eq:K-functional}) as well as relationships between
the $K-$functional and generalized Bochner-Riesz means introduced
in \cite{DD}. Our proofs given above are far simpler and in the vein
of classical Fourier analysis. It would be interesting find $K-$functional
relations for the operators of Propositions \ref{prop:binomial powers}
and \ref{prop:DD}. (The techniques used in \cite{DD} use the specific
structure of the multiplier associated with $V_{l}^{t}$ and do not
shed light on this problem; see also \cite{DI}.)
\end{remark}

\section{Estimates \& Integrability of Fourier Transform}

Herein we present applications of the theorems of the previous section.
Quantitative Riemann-Lebesgue estimates a deduced from the ``large
$\xi$'' part of the estimates given in Theorems \ref{thm:gen-estimates}
and \ref{thm:picks} and an integrability result is deduced from the
``small $\xi$'' part.

\subsection{Riemann-Lebesgue Type Estimates}

The following result is a general form of one given in \cite{BP1}
for the spherical mean operator. This result is immediate from the
estimates given in Theorem \ref{thm:gen-estimates}.
\begin{corollary}
Let $\mu\in\mathcal{K}_{\sigma}$ for some $\sigma>0$ and let $1\leq p\leq2$.
Then for any $f\in L^{p}(\mathbb{R}^{n})$:~
\begin{itemize}
\item when $p=1$,
\[
\sup_{|\xi|>1/t}\vert\widehat{f}(\xi)\vert\leq c_{1}\left\Vert M_{\mu}^{t}f-f\right\Vert _{1};
\]

\item when $1<p\leq2$,
\[
\int_{|\xi|>1/t}\vert\widehat{f}(\xi)\vert^{p'}d\xi\leq c_{p}\left\Vert M_{\mu}^{t}f-f\right\Vert _{p}^{p'}.
\]

\end{itemize}
\end{corollary}
The following variation is apparent from Theorem \ref{thm:picks};
in original form it appeared in \cite{D2,GT}.
\begin{corollary}
Let $\mu\in\mathcal{K}_{\sigma}$ for some $\sigma>0$, let $1<p\leq2$
and $p\leq q\leq p'$. Then for any $f\in L^{p}(\mathbb{R}^{n})$,
there is a constant $c_{p}>0$ such that
\[
\int_{|\xi|>1/t}\vert\xi\vert^{qn(1-\frac{1}{p}-\frac{1}{q})}\vert\widehat{f}(\xi)\vert^{q}d\xi\leq c_{p}\left\Vert M_{\mu}^{t}f-f\right\Vert _{p}^{q}.
\]

\end{corollary}
When $p=2$ the estimates in the two corollaries above are identical.
A specialization is possible in this case making use of a Lipschitz
condition. The following result generalizes one found in \cite{P}
and rediscovered in \cite{G}; the result in one dimension dates back
to Titchmarsh \cite[Theorem 85]{T}.
\begin{proposition}
Let $\mu\in\mathcal{K}_{\sigma}$ for some $\sigma>0$ and let $0<\alpha\leq2\sigma$.
Then for $f\in L^{2}(\mathbb{R}^{n})$,
\begin{equation}
\left\Vert M_{\mu}^{t}f-f\right\Vert _{2}=O(t^{\alpha})\;(t\rightarrow0)\label{eq:lip cond}
\end{equation}
if and only if
\begin{equation}
\int_{|\xi|>1/t}\vert\widehat{f}(\xi)\vert^{2}d\xi=O(t^{2\alpha})\;(t\rightarrow0).\label{eq:tail cond}
\end{equation}
\end{proposition}
\begin{proof}
That (\ref{eq:lip cond}) implies (\ref{eq:tail cond}) is immediate
from the previous corollaries. For the other implication we modify
the technique given in \cite{P}. Let
\[
F(\lambda)=\int_{S^{n-1}}\vert\widehat{f}(\lambda\omega)\vert^{2}d\omega,
\]
where $d\omega$ is the usual surface measure on $S^{n-1}$. From
the $L^{2}-$estimate in the preceding corollary,
\begin{align*}
\left\Vert M_{\mu}^{t}f-f\right\Vert _{2}^{2} & =O\left(\int_{1/t}^{\infty}F(\lambda)\lambda^{n-1}d\lambda+t^{4\sigma}\int_{0}^{1/t}\lambda^{4\sigma}F(\lambda)\lambda^{n-1}d\lambda\right)\\
 & =O(t^{2\alpha})+O\left(t^{4\sigma}\int_{0}^{1/t}\lambda^{4\sigma}F(\lambda)\lambda^{n-1}d\lambda\right).
\end{align*}
Now set $\varphi(r)=\int_{r}^{\infty}F(\lambda)\lambda^{n-1}d\lambda$,
then integration by parts and the hypothesis yield:
\begin{align*}
t^{4\sigma}\int_{0}^{1/t}\lambda^{4\sigma}F(\lambda)\lambda^{n-1}d\lambda & =-t^{4\sigma}\int_{0}^{1/t}\lambda^{4\sigma}\varphi'(\lambda)d\lambda\\
 & =-t^{4\sigma}\left[t^{-4\sigma}\varphi(1/t)-4\sigma\int_{0}^{1/t}\lambda^{4\sigma-1}\varphi(\lambda)d\lambda\right]\\
 & =O(t^{2\alpha})+O\left(t^{4\sigma}\int_{0}^{1/t}\lambda^{4\sigma-2\alpha-1}d\lambda\right)\\
 & =O(t^{2\alpha}).
\end{align*}
Putting the estimates together completes the proof.
\end{proof}

\subsection{An Integrabilty Result}

Let $f\in L^{p}(\mathbb{R}^{n})$ for some $1\leq p\leq2$, then in
spherical coordinates $\widehat{f}(\lambda\omega)$ is defined a.e.
and we set
\[
F(\lambda)=\left\{ \begin{array}{cc}
\max_{\omega\in S^{n-1}}\vert\widehat{f}(\lambda\omega)\vert, & p=1\\
\left[\int_{S^{n-1}}\vert\widehat{f}(\lambda\omega)\vert^{p'}d\omega\right]^{1/p'}, & 1<p\leq2
\end{array}\right..
\]
Below is a general integrability theorem concerning $F(\lambda)$.
The Lipschitz condition in the hypothesis has natural limitations
because of the nature of the inequalities in Theorem \ref{thm:gen-estimates}
for small $|\xi|$, e.g., in the case $1<p\leq2$, we have
\begin{equation}
\int_{|\xi|<1/t}|\xi|^{2\sigma p'}\vert\widehat{f}(\xi)\vert^{p'}d\xi\leq c_{p}^{p'}\left(\frac{\Vert M_{\mu}^{t}f-f\Vert_{p}}{t^{2\sigma}}\right)^{p'}.\label{eq:small xi}
\end{equation}

\begin{proposition}
\label{prop:integrab result}Let $\mu\in K_{\sigma}$ and $f\in L^{p}(\mathbb{R}^{n})$
for some $1\leq p\leq2$. If for some $0<\alpha\leq2\sigma$,
\[
\Vert M_{\mu}^{t}f-f\Vert_{p}=O(t^{\alpha}),\; t\rightarrow0,
\]
then $F\in L^{\beta}((0,\infty),\,\lambda^{n-1}d\lambda)$ provided
\begin{equation}
\frac{np}{np+\alpha p-n}<\beta\leq p'.\label{eq:integrab-est}
\end{equation}
\end{proposition}
\begin{proof}
Set $\Lambda=t^{-1}$, and rewrite (\ref{eq:small xi}) as,
\[
\int_{1}^{\Lambda}\lambda^{2\sigma p'}F(\lambda)^{p'}\lambda^{n-1}d\lambda\leq C\,\Lambda^{(2\sigma-\alpha)p'}.
\]
Take $\beta<p'$ and let
\[
\phi(\Lambda)=\int_{1}^{\Lambda}\left[\lambda^{2\sigma}F(\lambda)\right]^{\beta}\lambda^{n-1}d\lambda.
\]
Applying Holder's inequality we deduce ($C$ designates a generic
constant which may change from expression to expression):
\begin{align*}
\phi(\Lambda) & =\int_{1}^{\Lambda}\left[\lambda^{2\sigma}F(\lambda)\right]^{\beta}\lambda^{(n-1)\frac{\beta}{p'}}\lambda^{(n-1)(1-\frac{\beta}{p'})}d\lambda\\
 & \leq C\left(\int_{1}^{\Lambda}\lambda^{2\sigma p'}F(\lambda)^{p'}\lambda^{n-1}d\lambda\right)^{\beta/p'}\left(\int_{1}^{\Lambda}\lambda^{n-1}d\lambda\right)^{1-\beta/p'}\\
 & \leq C\,\Lambda^{2\sigma\beta-\alpha\beta}\left(1+\Lambda^{n}\right)^{1-\frac{\beta}{p'}}\\
 & =O\left(\Lambda^{2\sigma\beta+n-\alpha\beta-\frac{n\beta}{p'}}\right).
\end{align*}
Integration by parts yields the identity,
\begin{align*}
\int_{1}^{\Lambda}F(\lambda)^{\beta}\lambda^{n-1}d\lambda & =\int_{1}^{\Lambda}\lambda^{-2\sigma\beta}\phi'(\lambda)d\lambda\\
 & =\Lambda^{-2\sigma\beta}\phi(\Lambda)+2\beta\sigma\int_{1}^{\Lambda}\lambda^{-2\sigma\beta-1}\phi(\lambda)d\lambda.
\end{align*}
The first term can be estimated from the estimate of $\phi$; for
the second term we have
\[
\int_{1}^{\Lambda}\lambda^{-2\sigma\beta-1}\phi(\lambda)d\lambda\leq C\,\int_{1}^{\Lambda}\lambda^{-\alpha\beta+n-\frac{n\beta}{p'}-1}d\lambda=O(1+\Lambda^{n-\alpha\beta-\frac{n\beta}{p'}}).
\]
Putting the estimates together and expressing the estimate in terms
of $p$ yields
\[
\int_{1}^{\Lambda}F(\lambda)^{\beta}\lambda^{n-1}d\lambda=O(1+\Lambda^{n-\alpha\beta-n\beta+\frac{n\beta}{p}})
\]
and it follows that $F\in L^{\beta}((0,\infty),\lambda^{n-1}d\lambda)$
provided $n-\alpha\beta-n\beta+\frac{n\beta}{p}<0$. The proof is
complete as the conditions on $\beta$ are equivalent to (\ref{eq:integrab-est}).
\end{proof}
The role of $\sigma$ in this result simply specifies the possible
range for the Lipschitz order $\alpha$ and the latter limits the
range on $\beta$, in particular the lower bound. The following corollary
gives the full higher dimensional generalization of $L^{1}-$integrability
part of Theorem \ref{prop:B-1d-int}.
\begin{corollary}
\label{cor:ae FT}Let $1\leq p\leq2$ and suppose $\mu\in\mathcal{K}_{\sigma}$
for $2\sigma>\frac{n}{p}$. If $f\in L^{p}(\mathbb{R}^{n})$ and $\Vert M_{\mu}^{t}f-f\Vert_{p}=O(t^{\alpha})$
for some $\frac{n}{p}<\alpha\leq2\sigma$, then $\widehat{f}\in L^{1}(\mathbb{R}^{n})$
and Fourier inversion holds a.e.
\end{corollary}
The proof of this is simply to observe that under the stated conditions,
the lower limit for $\beta$ in (\ref{eq:integrab-est}) is less than
one. Hence $F\in L^{1}((0,\infty),\lambda^{n-1}d\lambda)$ and the
result follows since $\Vert\widehat{f}(\lambda\cdot)\Vert_{L^{1}(S^{n-1})}\leq\omega_{n-1}\Vert\widehat{f}(\lambda\cdot)\Vert_{L^{p}(S^{n-1})}$.

Explicit realizations of this corollary would be with the measures
defined in Propositions \ref{prop:binomial powers} and \ref{prop:DD}
for with $2l>\frac{n}{p}$. As stated in Remark \ref{connection-with-approximation},
examples of these measures and associated operators are connected
with higher order smoothness. Hence the implication from the above
corollary is that dimension dependent higher order smoothness conditions
are needed to achieve a.e. Fourier inversion. This is in line with
other works, e.g., Pinsky's work on Fourier inversion at a point \cite{Pinsky1,Pinsky2}.

The following example indicates that the range for $\beta$ in Proposition
\ref{prop:integrab result} is best possible; this example is a based
on one used by Titchmarsh \cite{T} in one dimension.
\begin{example}
Let $\mu$ be a surface measure satisfying the hypothesis of Proposition
\ref{prop:surf-meas}. Let $1<p\leq2$ and take $\frac{n}{p}-1<\gamma<\frac{n}{p}$.
Consider the radial function
\[
f(x)=\frac{1}{|x|^{\gamma}+|x|^{n}}.
\]
Then $f\in L^{p}(\mathbb{R}^{n})$. \textbf{Claim:} $\Vert M_{\mu}^{t}f-f\Vert_{p}=O\left(t^{\frac{n}{p}-\gamma}\right)$.
Once the claim is proved, it follows that $F\in L^{\beta}((0,\infty),\lambda^{n-1}d\lambda)$
for $\frac{n}{n-\gamma}<\beta\leq p'$. On the other hand, as $f$
is a radial function (and abusing notation), $F(\lambda)=\widehat{f}(\lambda)$
and since $f(x)\sim|x|^{-\gamma}$ for $x\rightarrow0$, it follows
that $F(\lambda)\sim\lambda^{\gamma-n}$ as $\lambda\rightarrow\infty$.
Thus $F(\lambda)^{\frac{n}{n-\gamma}}\lambda^{n-1}\sim\lambda^{-1}$
and $F\notin L^{\frac{n}{n-\gamma}}((0,\infty),\lambda^{n-1}d\lambda)$.\\
In order to prove the claim, let $d$ be the diameter of the compact
set $K$, let $t<1/d$, and consider $x$ such that $|x|>dt$. Using
the mean value theorem, $\vert M_{\mu}^{t}f(x)-f(x)\vert\leq t\vert\mathbf{\nabla}f(x)\vert$.
Write $f(x)=\tilde{f}(r)$, $r=|x|$, then
\[
\vert\mathbf{\nabla}f(x)\vert=\vert\tilde{f}'(r)\vert=\frac{\gamma r^{\gamma-1}+nr^{n-1}}{\left(r^{\gamma}+r^{n}\right)^{2}}.
\]
It follows that
\[
\int_{|x|>dt}\vert M_{\mu}^{t}f(x)-f(x)\vert^{p}dx\leq\omega_{n-1}t^{p}\int_{dt}^{\infty}\vert\tilde{f}'(r)\vert^{p}r^{n-1}dr=I_{1}(t)+I_{2},
\]
where $I_{1}(t)$ is the integral over $[dt,1]$ and $I_{2}$ over
the interval $[1,\infty)$. Then letting $C$ denote a generic constant,
not necessarily the same in each occurrence, we have
\begin{eqnarray*}
I_{1}(t) & \leq & C\int_{dt}^{1}r^{-p\gamma-p-n-1}dr=O\left(1+t^{n-p\gamma-p}\right)\\
I_{2} & \leq & C\int_{1}^{\infty}r^{n-np-p-1}dr=O(1),
\end{eqnarray*}
the latter because $n-np-p<0$. Putting the pieces together yields
the estimate
\[
\left(\int_{|x|>dt}\vert M_{\mu}^{t}f(x)-f(x)\vert^{p}dx\right)^{1/p}=O\left(t^{p}+t^{\frac{n}{p}-\gamma}\right)=O\left(t^{\frac{n}{p}-\gamma}\right),
\]
the last estimate as $\frac{n}{p}-1<\gamma<\frac{n}{p}$. To finish
the proof of the claim, consider the estimate
\begin{multline*}
\left(\int_{|x|\leq dt}\vert M_{\mu}^{t}f(x)-f(x)\vert^{p}dx\right)^{1/p}\leq\left(\int_{|x|\leq dt}\vert M_{\mu}^{t}f(x)\vert^{p}dx\right)^{1/p}+\left(\int_{|x|\leq dt}\vert f(x)\vert^{p}dx\right)^{1/p}\\
=J_{1}(t)+J_{2}(t).
\end{multline*}
Straightforward estimates show that $J_{2}(t)=O\left(t^{\frac{n}{p}-\gamma}\right)$
as $t\rightarrow0$. To estimate the first piece we use Minkowski's
inequality to write:
\[
J_{2}(t)\leq\int_{S}\left(\int_{|x|\leq dt}\vert f(x-ty)\vert^{p}dx\right)^{1/p}d\mu(y)=\int_{S}\left(\int_{|z-ty|\leq dt}\vert f(z)\vert^{p}dz\right)^{1/p}d\mu(y).
\]
Now $\{z\,:\,|z-ty|\leq dt\}\subset\{z\,:\,|z|\leq2dt\}$ and hence
\[
J_{2}(t)\leq\int_{S}\left(\int_{|z|\leq2dt}\vert f(z)\vert^{p}dz\right)^{1/p}d\mu(y)=O\left(t^{\frac{n}{p}-\gamma}\right).
\]
This completes the proof of the claim.
\end{example}

\section{Further Results}

In this section two variations on the results presented above for
the Euclidean Fourier transform are given. The first lies outside
the realm of multipliers used above and concerns a variant for weak
solutions of the Cauchy problem for the wave equation. Secondly, transference
results are used to push the results of section 2 into the realm of
the $n-$dimensional torus.

\subsection{The Cauchy problem for the wave equation}

Let $f\in L^{2}(\mathbb{R}^{n})$ and let $u(x,t)$ be the weak solution
of the Cauchy problem
\begin{align*}
\text{PDE: } & u_{tt}=\Delta u,\, x\in\mathbb{R}^{n},\, t>0\\
\text{IC: } & \left\{ \begin{array}{c}
u(x,0)=0\\
u_{t}(x,0)=f(x)
\end{array},\, x\in\mathbb{R}^{n}\right.
\end{align*}
On the Fourier transform side, $\widehat{u}(\xi,t)=\frac{\sin t|\xi|}{|\xi|}\widehat{f}(\xi)$.
Since $\mathcal{W}(t|\xi|)=\frac{\sin t|\xi|}{t|\xi|}$ is a bounded
function, it forms a bounded multiplier operator on $L^{2}(\mathbb{R}^{n})$.
Moreover, it is easily shown that
\[
\left\Vert \frac{u(\cdot,t)}{t}-f(\cdot)\right\Vert _{2}\rightarrow0,\, t\rightarrow0.
\]
Note that $\mathcal{W}(r)=j_{1/2}(r)$, and hence $|1-\mathcal{W}(r)|\asymp\min\{1,r^{2}\}$.
Using the techniques from section 2 leads to the following proposition.
\begin{proposition}
Let $f\in L^{2}(\mathbb{R}^{n})$ and let $u(x,t)$ be the weak solution
of the Cauchy problem above. Then
\[
\left[\int_{\mathbb{R}^{n}}\min\{1,(t|\xi|)^{4}\}\vert\widehat{f}(\xi)\vert^{2}d\xi\right]^{1/2}\asymp\left\Vert \frac{u(\cdot,t)}{t}-f(\cdot)\right\Vert _{2},
\]
and moreover
\[
\left\Vert \frac{u(\cdot,t)}{t}-f(\cdot)\right\Vert _{2}\asymp\left\Vert M^{t}f(\cdot)-f(\cdot)\right\Vert _{2}.
\]

\end{proposition}
The second conclusion follows from the first and the $L^{2}-$result
in Theorem \ref{BP-thm}; see also Example \ref{approx-ident-examp}.
\begin{remark}
In the case $n=3$, $\frac{u(x,t)}{t}=M^{t}f(x)$, the spherical mean
operator. This suggests that the proposition should have an generalization
to other $L^{p}-$spaces. The difficulty is that the function $\mathcal{W}(t|\xi|)$
is not a bounded multiplier on all $L^{p}(\mathbb{R}^{n})$. Rather,
by embedding the spherical mean operator into an analytic family of
operators, Stein \cite{S1} showed for $n\geq4$ and $\frac{2n}{n+1}<p\leq2$,
$\mathcal{W}(t|\xi|)$ is a bounded multiplier on $L^{p}(\mathbb{R}^{n})$
and moreover
\[
\left\Vert \frac{u(\cdot,t)}{t}-f(\cdot)\right\Vert _{p}\rightarrow0,\, t\rightarrow0.
\]
Under these conditions we then obtain the following inequality
\[
\left[\int_{\mathbb{R}^{n}}\min\{1,(t|\xi|)^{2p'}\}\vert\widehat{f}(\xi)\vert^{p'}d\xi\right]^{1/p'}\leq c_{p}\left\Vert \frac{u(\cdot,t)}{t}-f(\cdot)\right\Vert _{p},
\]
for some positive constant $c_{p}$. This result is also valid for
$n=1,2,3$ without restriction on $p$, i.e.. for $1\leq p\leq2$.
\end{remark}

\subsection{Transference to the torus}

In this section the use of the hat notation $\widehat{f}$ will be
used to denote Fourier coefficients as well as Fourier transforms;
the meaning should be clear from context. Let $\mathbb{T}$ be the
unit circle in the plane, $f\in L^{p}(\mathbb{T})$, and consider
the following modulus of continuity:
\[
\omega_{p}[f](t)=\sup_{0<h<t}\Vert f(\cdot+h)+f(\cdot-h)-2f(\cdot)\Vert_{L^{p}(\mathbb{T})}.
\]
The analog of Theorem \ref{thm:BP-1d} in the context of Fourier series
takes the following form.
\begin{theorem}
Let $1\leq p\leq2$. Then there is a constant $c_{p}>0$ such that
for all $f\in L^{p}(\mathbb{T})$ :
\begin{itemize}
\item when $p=1$,
\[
\sup_{k\in\mathbb{Z}}\left[\min\{1,(t|k|)^{2}\}\vert\widehat{f}(k)\vert\right]\leq c_{1}\omega_{1}[f](t);
\]

\item when $1<p<2$,
\[
\left[\sum_{k\in\mathbb{Z}}\min\{1,(t|k|)^{2p'}\}\vert\widehat{f}(k)\vert^{p'}\right]^{1/p'}\leq c_{p}\omega_{p}[f](t);
\]

\item when $p=2$,
\[
\left[\sum_{k\in\mathbb{Z}}\min\{1,(t|k|)^{4}\}\vert\widehat{f}(k)\vert^{2}\right]^{1/2}\asymp\omega_{2}[f](t).
\]

\end{itemize}
\end{theorem}
The proof of this follows the ideas/methods given in section 2. Analogous
to Proposition \ref{prop:B-1d-int}, the following corollary can be
proved adapting the ideas of section 3.
\begin{corollary}
\label{cor:szasz}Let $1\leq p\leq2$ and $f\in L^{p}(\mathbb{T})$
such that for some $0<\alpha\leq2$,
\[
\Vert f(\cdot+t)+f(\cdot-t)-2f(\cdot)\Vert_{L^{p}(\mathbb{T})}=O(t^{\alpha}).
\]
Then $\sum_{k}\vert\widehat{f}(k)\vert^{\beta}$ converges for $\frac{p}{p+\alpha p-1}<\beta\leq p'$.
Moreover, if the above Lipschitz condition holds for some $\alpha>1/p$,
then the Fourier series of $f$ converges absolutely and uniformly
on $\mathbb{T}$.
\end{corollary}
Looking toward extensions to higher dimension, let $\mathbb{T}^{n}=\mathbb{R}^{n}/\mathbb{Z}^{n}$
be the torus in $n-$dimensions. The measures $\mu\in\mathcal{K}_{\sigma}(\mathbb{R}^{n})$
define Fourier multipliers on $L^{p}(\mathbb{R}^{n})$ and can be
transferred to Fourier multipliers on $L^{p}(\mathbb{T}^{n})$ under
mild regularity conditions on $\widehat{\mu}$ (see \cite[Chapter 3]{Graf}).
Specifically, if $\widehat{\mu}$ is continuous on $\mathbb{R}^{n}$
(all of the examples given in this paper satisfy this condition),
then the series, 
\[
\sum_{k\in\mathbb{Z}^{n}}\widehat{\mu}(k)e^{ik\cdot x},
\]
is convergent to a measure, also denoted $\mu$ on $\mathbb{T}^{n}$.
This class of transferred measures/multipliers is denoted $\mathcal{K}_{\sigma}(\mathbb{T}^{n})$.
Further, $\widehat{\mu}(t\, k)$, $t>0$ and $k\in\mathbb{Z}^{n}$,
defines a Fourier multiplier operator on $L^{p}(\mathbb{T}{}^{n})$.
The corresponding operator, also denoted $M_{\mu}^{t}$ is given by
\[
M_{\mu}^{t}f(x)=\sum_{k\in\mathbb{Z}^{n}}\widehat{\mu}(t\, k)\widehat{f}(k)e^{ik\cdot x},\; f\in L^{p}(\mathbb{T}^{n}).
\]
Here $\{\widehat{f}(k)\}_{k\in\mathbb{Z}^{n}}$ is the sequence of
Fourier coefficients of $f$ and the series is convergent in $L^{p}-$norm.
Given this set of ideas, the following variation of Theorem \ref{thm:gen-estimates}
is apparent.
\begin{proposition}
Let $n\geq2$, $\mu\in\mathcal{K}_{\sigma}(\mathbb{T}^{n})$, and
let $1\leq p\leq2$. Then there is a constant $c_{p}>0$ such that
for all $f\in L^{p}(\mathbb{T}^{n})$:~
\begin{itemize}
\item when $p=1$,
\[
\sup_{k\in\mathbb{Z}^{n}}\left[\min\{1,(t\,|k|)^{2\sigma}\}\vert\widehat{f}(k)\vert\right]\leq c_{1}\left\Vert M_{\mu}^{t}f-f\right\Vert _{1};
\]

\item when $1<p<2$,
\[
\left[\sum_{k\in\mathbb{Z}^{n}}\min\{1,(t\vert k\vert)^{2\sigma p'}\vert\widehat{f}(k)\vert^{p'}\right]^{1/p'}\leq c_{p}\left\Vert M_{\mu}^{t}f-f\right\Vert _{p};
\]

\item when $p=2$,
\[
\left[\sum_{k\in\mathbb{Z}^{n}}\min\{1,(t\vert k\vert)^{4\sigma}\vert\widehat{f}(k)\vert^{2}\right]^{1/2}\asymp\left\Vert M_{\mu}^{t}f-f\right\Vert _{2}.
\]

\end{itemize}
\end{proposition}
A version of this result may also be made in the vein of Theorem \ref{thm:picks}
using Pick's inequality on $\mathbb{T}^{n}$; this generalizes a result
in \cite{GT} on the torus.

For $k\in\mathbb{Z}^{n}$, let $\Vert k\Vert=\max_{1\leq j\leq n}k_{j}$
be the maximum norm. Due to the equivalence of norms on finite dimensional
spaces, all of the estimates in the proposition above can be rewritten
in terms of this norm. This fact will be useful in the proof of the
following corollary generalizing Corollary \ref{cor:szasz} and in
the vein of Proposition \ref{prop:integrab result}.
\begin{corollary}
\label{cor:ae FS}Let $n\geq2$, let $\mu\in\mathcal{K}_{\sigma}(\mathbb{T}^{n})$
for some $\sigma>0$, and let $1\leq p\leq2$. If $f\in L^{p}(\mathbb{T}^{n})$
such that for some $0<\alpha\leq2\sigma$,
\[
\left\Vert M_{\sigma}^{t}f-f\right\Vert _{L^{p}(\mathbb{T}^{n})}=O(t^{\alpha}),
\]
then $\sum_{k}\vert\widehat{f}(k)\vert^{\beta}$ converges for $\frac{np}{np+\alpha p-n}<\beta\leq p'$.
Moreover, if $2\sigma>\frac{n}{p}$ and the Lipschitz condition holds
for some $\alpha>\frac{n}{p}$, then $\sum_{k\in\mathbb{Z}^{n}}\vert\widehat{f}(k)\vert$
converges, the Fourier series of $f$ converges absolutely and uniformly.
\end{corollary}
\begin{proof}
The proof is similar to that of Proposition \ref{prop:integrab result}
as follows. We consider the case $1<p\leq2$, and take $N\in\mathbb{N}$.
Then the estimate in the proposition above implies
\[
\sum_{\Vert k\Vert\leq N}\Vert k\Vert^{2\sigma p'}\vert\widehat{f}(k)\vert^{p'}\leq C\, N^{(2\sigma-\alpha)p'},
\]
for some constant $C>0$. For $0<\beta<p'$, this implies via H\={o}lder's
inequality
\[
\sum_{\Vert k\Vert\leq N}\Vert k\Vert^{2\sigma\beta}\vert\widehat{f}(k)\vert^{\beta}=O\left(N^{2\sigma\beta+n-\alpha\beta-\frac{n\beta}{p'}}\right).
\]
Let 
\[
\phi(N)=\sum_{1\leq\Vert k\Vert\leq N}\Vert k\Vert^{2\sigma\beta}\vert\widehat{f}(k)\vert^{\beta}=\sum_{l=1}^{N}l^{2\sigma\beta}\sum_{\Vert k\Vert=l}\vert\widehat{f}(k)\vert^{\beta},
\]
and $\phi(0)=0$, then applying summation by parts
\begin{eqnarray*}
\sum_{1\leq\Vert k\Vert\leq N}\vert\widehat{f}(k)\vert^{\beta} & = & \sum_{l=1}^{N}l^{-2\sigma\beta}(\phi(l)-\phi(l-1))\\
 & = & \sum_{l=1}^{N-1}\left(l^{-2\sigma\beta}-(l+1)^{-2\sigma\beta}\right)\phi(l)+N^{-2\sigma\beta}\phi(N)\\
 & = & I_{1}(N)+I_{2}(N).
\end{eqnarray*}
The second piece is easily estimated using the estimate on $\phi$
to obtain $I_{2}(N)=O\left(N^{n-\alpha\beta-\frac{n\beta}{p'}}\right).$
For the first piece, we use the estimate on $\phi$ and simple estimates:
\begin{align*}
I_{1}(N) & \leq C\,\sum_{l=1}^{N}l^{-2\sigma\beta}\left(1-\left(1+\frac{1}{l}\right)^{-2\sigma\beta}\right)l^{2\sigma\beta+n-\alpha\beta-\frac{n\beta}{p'}}\\
 & \leq C\,\sum_{l=1}^{N}l^{n-\alpha\beta-\frac{n\beta}{p'}-1}\\
 & =O\left(1+N^{n-\alpha\beta-\frac{n\beta}{p'}}\right)
\end{align*}
Putting the estimates together we have
\[
\sum_{1\leq\Vert k\Vert\leq N}\vert\widehat{f}(k)\vert^{\beta}=O\left(1+N^{n-\alpha\beta-n\beta+\frac{n\beta}{p}}\right),
\]
and the result follows.
\end{proof}
Applying transference to the measures from Propositions \ref{prop:binomial powers}
and \ref{prop:DD} for $2l>\frac{n}{p}$ yields explicit realizations
for the second conclusion of this result.

\end{document}